\documentclass{article}
\usepackage{tikz}
\usepackage{graphicx}
\usepackage{amsmath}
\usepackage{amsthm}
\usepackage{amssymb}
\usepackage{setspace}
\usepackage{latexsym}
\usepackage{setspace}

\usepackage[left=3cm]{geometry}

\usepackage{amsfonts}

\usepackage{comment}
\providecommand{\keywords}[1]{\textbf{Keywords:} #1}
\newtheorem{lemma}{Lemma}[section]
\newtheorem{proposition}{Proposition}[section]
\newtheorem{remark}{Remark}[section]
\newtheorem{example}{Example}[section]

\title{Kaminsky Type Functional Equations and Bivariate Residual Lifetimes Distributions}
\author{Sabrina Mulinacci, Massimo Ricci\\Department of Statistical Sciences \\ University of Bologna}

\begin{document}

\maketitle
\begin{abstract}
This paper considers generalizations of the functional equations that characterize the lack-of-memory properties at univariate and bivariate levels. Specifically, we extend the univariate functional equation introduced by Kaminsky (1983) (that characterizes the Gompertz distribution) and the corresponding bivariate strong and weak versions later studied in Marshall and Olkin (2015) by allowing the conditional survival distribution to be a fully general time dependent distortion of the unconditional one: in particular, we show that the solutions of these generalized functional equations coincide with the solutions of the functional equations studied in Ricci (2024). Since the univariate functional equation leads only to a trivial case and the solutions of the strong bivariate functional equation have been already studied in the literature, the analysis is focused on the weak bivariate case, where joint residual lifetimes are conditioned on survival beyond a common threshold $t$. In view of potential applications to insurance risk analysis, the impact of the time dependent distortion on the aging properties of the associated distribution is analized as well as the time dependent dependence structure of the residual lifetimes through time-varying versions of the Kendall’s function and of the tail dependence coefficients. Many examples are provided and a wide family of bivariate survival distributions satisfying the generalized weak functional equation is constructed through a mixing approach.  
\end{abstract}
\section{Introduction}
In Marshall and Olkin (1967), the authors introduce bivariate extensions of the functional equation that characterizes the lack-of-memory property for random variables, that is
\begin{equation}
\Bar{G}(x+t) = \Bar{G}(x)\cdot \Bar{G}(t) , \quad x,t\geq 0,
\label{Univariate LMP}
\end{equation}
where $\bar G$ is the survival distribution function of a non-negative random variable. They consider the bidimensional functional equations
\begin{equation}
    \bar{G}(s_1+t_1,s_2+t_2) = \bar{G}(s_1,s_2)\cdot \Bar{G}(t_1,t_2) , \quad s_1, s_2, t_1, t_2 \geq 0,
    \label{strong standard equation}
\end{equation}
and
\begin{equation}
 \bar{G}(s_1+t,s_2+t) = \bar{G}(s_1,s_2)\cdot \Bar{G}(t,t), \quad s_1, s_2, t \geq 0,
    \label{weak standard equation}
\end{equation}
where $\bar G$ is the survival distribution function of a continuous and non-negative bidimensional random vector: in the first case, $\bar G$ is said to satisfy the {\it strong bivariate lack-of-memory property}, while in the second one, $\bar G$ is said to satisfy the {\it weak lack-of-memory property}.

\noindent The unique solution of (\ref{strong standard equation}) is given by the family 
\begin{equation}\label{sol_LP_strong}\Bar G(x,y)=e^{-\lambda_1x-\lambda_2 y},\quad x,y\geq0,\end{equation}
with $\lambda_1,\lambda_2>0$, while that of (\ref{weak standard equation}) 
is given by the family 
\begin{equation}
    \begin{split}
     \Bar{G}(x,y)    & = \begin{cases}
        e^{-\lambda y} \Bar{G}_1(x-y) \ x \geq y\geq 0 \\
        e^{-\lambda x} \Bar{G}_2(y-x) \ 0\leq x < y
    \end{cases},
    \end{split}
    \label{PWBLMP_eq}
\end{equation}
where $\Bar{G}_i, \ i = 1,2$ are marginal survival functions and $\lambda >0$. Indeed, (\ref{PWBLMP_eq}) is a bivariate survival distribution function with marginal absolutely continuous distributions with densities $g_i(z)=-\Bar G^\prime_i(z)$, $i=1,2$, if and only if $0<\lambda\leq g_1(0)+g_2(0)$ and $\frac d{dz}\log g_i(z)\geq -\lambda$, for all $z\geq 0$ and $i=1,2$. Moreover, if $(X,Y)$ is distributed according to (\ref{PWBLMP_eq}), then $\mathbb P\left (X=Y\right)=\frac{g_1(0)+g_2(0)}\lambda-1$.
\vskip 0.3cm

A particular generalization of the functional equation (\ref{Univariate LMP}) has been considered in Kaminsky (1983):
\begin{equation}
    \frac{\Bar{F}(x+t) }{\Bar{F}(t)}= (\Bar{F}(x))^{\phi(t)}, \quad x,t \geq 0,
    \label{Kamisnky Functional}
\end{equation}
with $\phi: [0,\infty) \rightarrow [0,\infty)$. This functional equation is satisfied by a proper survival distribution function if and only if it coincides with the Gompertz distribution and $\phi(t)=e^{\lambda t}$ for $\lambda>0$. 

In Marshall and Olkin (2015), two bivariate versions of (\ref{Kamisnky Functional}) have been considered. More precisely, they analyze the {\it strong bivariate functional equation} 
\begin{equation}
   \frac{ \Bar{F}(x+s,y+t)}{\Bar{F}(s,t)} = (\Bar{F}(x,y))^{\phi(s,t)}, \quad x,y,s,t \geq 0,
    \label{strong functional}
\end{equation}
with $\phi:[0,\infty) \times [0,\infty) \rightarrow [0,\infty)$ and the {weak bivariate functional equation}
\begin{equation}
   \frac{ \Bar{F}(x+t,y+t)}{\Bar{F}(t,t)} = (\Bar{F}(x,y))^{\phi(t)}, \quad x,y,t \geq 0,
    \label{weak_functional}
\end{equation}
with $\phi: [0,\infty) \rightarrow [0,\infty)$. The functional equations (\ref{Kamisnky Functional}), (\ref{strong functional}) and (\ref{weak_functional}) generalize the corresponding lack-of-memory properties (\ref{Univariate LMP}), (\ref{strong standard equation}) and (\ref{weak standard equation}) that can be recovered by setting $\phi(t)=1$ and $\phi(s,t)=1$, depending on the case.

In Marshall and Olkin (2015) the authors show that, under the assumption that marginal distributions are of Gompertz type, 
the solution of (\ref{strong functional}) is given by 
\begin{equation*}
    \bar F(x,y) = e^{-\xi(e^{\lambda_1 x + \lambda_2 y}-1)}
\end{equation*}
for some $\lambda_1, \lambda_2 > 0$ and for some $\xi \geq 1$ (see also Kolev, 2016, for alternative characterizations of these distributions), while the solution of (\ref{weak_functional}) is given by
\begin{equation}\label{MO2015}
    \bar{F}(x,y) = 
    \begin{cases}
       e^{-\xi (e^{\lambda y}-1)-e^{\lambda y} \xi_1 (e^{\lambda_1 (x-y)}-1)}, \ x \geq y \\
       e^{-\xi (e^{\lambda x}-1)-e^{\lambda x} \xi_2 (e^{\lambda_2 (y-x)}-1)}, \ x < y \\
    \end{cases}
\end{equation}
for some $\lambda,\lambda_1,\lambda_2,\xi,\xi_1,\xi_2\geq 0$, with $\lambda \geq \max(\lambda_1, \lambda_2)$, $\lambda (\xi-1) \geq \max (\lambda_1(\xi_1-1), \lambda_2(\xi_2-1))$ and 
$\lambda_1 \xi_1 + \lambda_2 \xi_2 \geq \lambda \xi$.\\

In existing literature, further extensions of the bivariate lack-of-memory properties have been considered and studied, introducing associative operators that generalize the standard sum operator. 
In Muliere and Scarsini (1986), 
the functional equations defining the lack-of-memory properties (at univariate as well at bivariate level) have been generalized by replacing the standard sum with a more general associative operator (see also Rao, 2004, for an alternative type of generalization in the same line of approach).
In Ricci (2024), a generalization of (\ref{Univariate LMP}), (\ref{strong standard equation}) and (\ref{weak standard equation}) is obtained by substituting the standard product with
the binary associative operator $\otimes_h$ in $[0,1]$, defined as
\begin{equation*}
a \otimes_h b = h(h^{-1}(a) \ h^{-1} (b)),\quad a,b\in[0,1],
\label{pseudo product}
\end{equation*}
where $h$ a strictly increasing bijection of $[0,1]$; the obtained functional equations are called {\it pseudo univariate}, {\it pseudo strong bivariate} and {\it pseudo weak bivariate} lack-of-memory functional equations, respectively.
\vskip 0.5cm
In this paper, following the same line of approach of Kaminsky (1983) in the one dimensional case and that of Marshall and Olkin (2015) in the bidimensional one, we consider the corresponding fully general functional equations
$$ \frac{\Bar{F}(x+t) }{\Bar{F}(t)}= d_t\left (\Bar{F}(x)\right), \quad x,t \geq 0,$$
$$\frac{ \Bar{F}(x+s,y+t)}{\Bar{F}(s,t)} = d_{s,t}\left(\Bar{F}(x,y)\right), \quad x,y,s,t \geq 0,$$
and
$$\frac{ \Bar{F}(x+t,y+t)}{\Bar{F}(t,t)} = d_t\left(\Bar{F}(x,y)\right), \quad x,y,t \geq 0,$$
where, for every $t\geq 0$, $d_t$ is a strictly increasing bijection of $[0,1]$ with $d_0(x)=x$ and, for all $s,t\geq 0$, $d_{s,t}$ is a strictly increasing bijection of $[0,1]$ with $d_{0,0}(x)=x$.
\vskip 0.3cm

As a main result, we will prove that these functional equations are equivalent to the {\it pseudo univariate}, to the {\it strong bivariate} and to the {\it pseudo weak bivariate} lack-of-memory functional equations of Ricci (2024), respectively, for suitable choices of $h$, $d_t$ and $d_{s,t}$. Since the one dimensional case results in a trivial solution and the strong one is solved by distributions already exhaustively studied in existing literature, we will focus our analysis on the weak case. Since this case represents a generalization of the weak bivariate lack-of-memory property that depends on the choice of $d_t$, we will analyze the dynamics of the dependence structure of the vector  $\mathbf{X}_t = (X-t,Y-t|X > t, Y > t)$ through the study of the time dependent Kendall's function and the tail dependence coefficients. Moreover, having in mind actuarial and reliability applications, we will analyze the bivariate aging properties induced again by $d_t$.

The paper is organized as follows. In Section \ref{prerequisites} we introduce the main concepts and results that are the base of the contents of the paper. In Section \ref{Kaminsky} we analyze fully general extensions of the Kaminsky functional equation, in the unidimensional case and of the Marshall-Olkin functional equations in the bidimensional case. 
Section \ref{Aging and generators} is devoted to the analysis of the aging properties induced by the considered functional equations while in Section \ref{Dependence} the time dependent dependence structure of the residual lifetimes is studied. Section \ref{conclusion} concludes.

\section{Preliminaries and notation}\label{prerequisites}
In this section we will fix the notation and introduce the main concepts that will be used in the paper.
\smallskip

Let $h:[0,1]\to[0,1]$ be a continuous and strictly increasing bijection. We call the binary operator $\otimes_h:[0,1]\times[0,1]\to[0,1]$, defined as
$$a\otimes_hb = h\left (h^{-1}(a)h^{-1}(b) \right ),\quad a,b\in[0,1],$$ 
 {\it pseudo product with generator $h$}. 

This operator generalizes the standard product (recovered when $h(x)=x$ for $x\in [0,1]$) and it is also known in the literature as "Archimedean T-norm": it is continuous, commutative, associative and strictly increasing in both arguments, so that $a\otimes_h1=a$ for all $a\in [0,1]$ and $a\otimes_ha < a$ for all $a \in (0,1)$ (see, for example, among the wide literature, Klement et al., 2004).
\smallskip

\noindent Notice that generators $h(x)$ and $h_\beta(x)=h(x^\beta), \beta > 0$ produce the same pseudo-product.
\vskip 0.3cm
In Ricci (2024), the pseudo-product is used to generalize the functional equations defining the different types of lack-of-memory properties, (\ref{Univariate LMP}), (\ref{strong standard equation}) and (\ref{weak standard equation}), by substituting the standard product with the pseudo-product, obtaining, respectively,
\begin{equation}\label{Univ_pseuso}
    \Bar{F}(x+t) = \Bar{F}(x) \otimes_h \Bar{F}(t),\quad\forall x,t\geq 0,
\end{equation}
\begin{equation}
    \Bar{F}(s_1+t_1, s_2 + t_2) = \Bar{F}(s_1, s_2) \otimes_h \Bar{F}(t_1,t_2),\quad \forall s_1,s_2,t_1,t_2 \geq 0
    \label{Strong Pseudo Equation}
\end{equation}
and 
\begin{equation}
    \Bar{F}(s_1+t, s_2 + t) = \Bar{F}(s_1, s_2) \otimes_h \Bar{F}(t,t), \quad\forall s_1,s_2,t \geq 0.
    \label{Weak Pseudo Equation}
\end{equation}
Moreover, in that paper, it is shown that the class of all the solutions of these functional equations are obtained by distorting through the generator $h$ the solutions of the corresponding standard lack-of-memory functional equations.
More precisely, the class of solution of (\ref{Univ_pseuso}) is given by 
\begin{equation*}
\Bar{F}(x) = h(e^{-\lambda x}), \ \lambda > 0,
\end{equation*}
that of  (\ref{Strong Pseudo Equation}) is of type (see (\ref{sol_LP_strong}))
\begin{equation}
         \Bar{F}(s,t) = h(\exp(-\lambda_1 s -\lambda_2 t)),\ \lambda_i \geq 0, \ i = 1,2,
         \label{PBLMP_eq}
\end{equation}
while the solution of (\ref{Weak Pseudo Equation}) is of type (see (\ref{PWBLMP_eq}))
\begin{equation}
     \Bar{F}(x,y)  = h(\Bar{G}(x,y)) =  \begin{cases}
        h\left(e^{-\lambda y} \Bar{G}_1(x-y)\right) \ x \geq y\geq 0 \\
        h\left(e^{-\lambda x} \Bar{G}_2(y-x)\right) \ 0\leq x < y
    \end{cases},
    \label{PWBLMPh_eq}
\end{equation}
where $\Bar{G}_i, \ i = 1,2$ are univariate marginal survival functions and $\lambda$ is a positive constant. 
\vskip 0.3cm
While (\ref{PBLMP_eq}) is proved to be a bivariate survival distribution function if and only if the function $h(e^{-x})$ is convex, determining conditions under which 
(\ref{PWBLMPh_eq}) is a bivariate survival function is not an easy task and the problem is widely discussed in Ricci (2024).
Moreover, it is shown that, if $h$ is twice differentiable with $h^\prime(x)>0$, for all $x\in [0,1]$, survival bivariate distributions of type (\ref{PWBLMPh_eq}) inherit from $\bar G$ satisfying (\ref{PWBLMP_eq}) the singular component, that is, if $(X,Y)$ is distributed according to (\ref{PWBLMPh_eq}),
\begin{equation}\label{sing}
\mathbb P(X=Y)=\frac{g_1 (0)+g_2 (0)}\lambda -1,\end{equation}
where $g_i^\prime=-\bar G^\prime _i$ for $i=1,2$: notice that, as a consequence, again, 
$\lambda\in(0,g_1^\prime (0)+g_2^\prime (0)]$. However,
 although $\mathbb P(X=Y)$ doesn't depend on the distortion $h$, the choice of $h$ determines how the singularity mass is spread on the straight line $x=y$. In fact, if $S(x)=\mathbb P(X=Y>x)$,
$S(x) =S(0)\cdot h\left (e^{-\lambda x}\right)$.
\smallskip

On the other side, survival distribution functions in (\ref{PBLMP_eq}) represent a type of generalization of the bivariate Schur-constant distributions that have been already studied in Genest and Kolev (2021): this is the reason why, in Ricci (2024), the analysis is focused on the weak case (\ref{PWBLMPh_eq}).
\smallskip

Obviously, the generators $h$ and $h_\beta$, given by $h_\beta(x)=h\left (x^\beta\right)$, define the same class of bivariate survival functions satifying (\ref{Univ_pseuso}), (\ref{Strong Pseudo Equation}) and (\ref{Weak Pseudo Equation}), for all $\beta>0$.

\vskip 0.3cm
\noindent In the whole paper we will consider continuous positive random variables with support $(0,+\infty)$. Moreover, given a positive random variable $X$ and a threshold $t>0$, we denote by $X_t$ the random variable 
\begin{equation}\label{survival_one}X_t=[X-t\vert X>t],\end{equation}
representing the excess of $X$ above $t$, called "residual lifetime".
Similarly, given the vector ${\bf X}=(X,Y)$ and given $s,t>0$, we denote by ${\bf X}_{s,t}$ the random vector  
\begin{equation}\label{survival_two_strong}{\bf X}_{s,t}=[X-s,Y-t\vert X>s,Y>t],\end{equation}
while, when $t=s$, we use the simplified notation
\begin{equation}\label{survival_two_weak}{\bf X}_{t}=[X-t,Y-t\vert X>t,Y>t].\end{equation}

\section{Generalized Kaminsky-type functional equations}\label{Kaminsky}
This Section is devoted to the main results of the paper that consist in determining the solutions of fully general extensions of Kaminsky and Marshall-Olkin functional equations. 
\subsection{The generalized Kaminsky functional equation}\label{Kaminsky1}
Here we start considering the generalization of the univariate Kaminsky functional equation (\ref{Kamisnky Functional}).

\noindent Let $X$ be distributed according to the survival distribution function $\bar F$. We analyze the functional equation (see (\ref{survival_one}))
\begin{equation}\label{Kaminsky_one}\bar F_{X_t}(x)=\frac{\bar F(x+t)}{\bar F(t)}=d_t\left (\bar F(x)\right ),\quad t,x \geq 0,\end{equation}
where, for every $t\geq 0$, $d_t$ is a strictly increasing bijection of $[0,1]$ with $d_0(x)=x$. (\ref{Kaminsky_one}) represents a generalization of Kaminsky's equation (\ref{Kamisnky Functional}) since the latter corresponds to the choice $d_t(x)=x^{\phi(t)}$. Obviously, the case $d_t(x)=x$ for all $t\geq 0$ identifies the standard lack-of-memory property case.
\vskip 0.3cm
\noindent From (\ref{Kaminsky_one}), we immediately get that $d_t$ is uniquely identified by
\begin{equation}\label{d_tuni}d_t(x)=\frac{\bar F\left(t+\bar F^{-1}(x)\right)}{\bar F(t)}.\end{equation}
\begin{remark}\label{re-Univ}
In Ricci (2024), it is proved that any univariate survival function $\Bar{F}$ satisfies (\ref{Univ_pseuso}) with $h(x) = \bar{F}(-\log x)$: notice that 
$d_t(x)=\frac{h\left (e^{-t}h^{-1}(x)\right )}{h\left (e^{-t}\right )}$.
\end{remark}

\subsection{The generalized strong bivariate Marshall-Olkin functional equation}\label{MO-S}
Let $(X,Y)$ be distributed according to the joint survival distribution function $\bar F$ and denote by $\bar F_{s,t}$ the joint survival distribution function of ${\bf X}_{s,t}$ (see (\ref{survival_two_strong}).
The Marshall and Olkin (2015) functional equation (\ref{strong functional}) can be generalized to 
\begin{equation}\label{MO_strong}\bar F_{s,t}(x,y)=\frac{\bar F(x+s,y+t)}{\bar F(s,t)}=d_{s,t}\left (\bar F(x,y)\right ),\quad t,x,s,y \geq 0,\end{equation}
where, for all $s,t\geq 0$, $d_{s,t}$ is a strictly increasing bijection of $[0,1]$ with $d_{0,0}(x)=x$.
\smallskip

Following the same reasoning as in the proof of Proposition 3.1 in Marshall and Olkin (2015), we get the following result:
\begin{proposition}\label{strong_strong}
A bivariate survival function $\bar F$ satisfies (\ref{MO_strong}) if and only if  there exist an univariate convex survival distribution function $\bar H$ and a constant $a > 0$ such that 
\begin{equation}\label{strong_sol}\bar{F}(x,y) = \bar H(x+ay)\end{equation}
and \begin{equation}\label{time_strong}d_{s,t}(x)=
\frac{\bar H\left (s+ta+\bar H^{-1}(x)\right )}{\bar H\left (s+ta\right )}.\end{equation}
\end{proposition}
\begin{proof}
From (\ref{strong_sol}) and (\ref{time_strong}), (\ref{MO_strong}) immediately follows.
\vskip 0.2cm
Conversely, let us assume that (\ref{MO_strong}) holds true for a bivariate survival function $\bar F$ and denote with $\bar F_1$ and $\bar F_2$ the two associated marginal survival distribution functions.
From (\ref{MO_strong}), setting $y=t=0$, we get 
$$\frac{\bar F_1(x+s)}{\bar F_1(s)}=d_{s,0}\left (\bar F_1(x)\right ),\quad x,s\geq 0,$$
while, setting $x=s=0$, we get
$$\frac{\bar F_2(y+t)}{\bar F_2(t)}=d_{0,t}\left (\bar F_2(y)\right ),\quad y,t\geq 0$$
and (see Subsection \ref{Kaminsky1}) 
\begin{equation}\label{two_dist}d_{s,0}(z)=\frac{\bar F_1\left(s+\bar F_1^{-1}(z)\right)}{\bar F_1(s)}\quad\text{ and }\quad d_{0,t}(z)=\frac{\bar F_2\left(t+\bar F_2^{-1}(z)\right)}{\bar F_2(t)}.\end{equation}
Now, again from (\ref{MO_strong}), setting $x=t=0$, we get 
\begin{equation}\label{e1}\bar F(s,y)=d_{s,0}\left (\bar F_2(y)\right )\bar F_1(s),\quad s,y\geq 0,\end{equation}
while, setting $y=s=0$, we obtain
\begin{equation}\label{e2}\bar F(x,t)=d_{0,t}\left (\bar F_1(x)\right )\bar F_2(t),\quad x,t\geq 0.\end{equation}
Since the two expressions in (\ref{e1}) and (\ref{e2}) must coincide, we get
$$d_{s,0}\left (\bar F_2(t)\right )\bar F_1(s)=d_{0,t}\left (\bar F_1(s)\right )\bar F_2(t),\quad s,t\geq 0$$
that, by (\ref{two_dist}), gives
$$\bar F_1\left (s+\bar F_1^{-1}\left (\bar F_2(t)\right )\right )= \bar F_2\left (t+\bar F_2^{-1}\left (\bar F_1(s)\right )\right ).$$
Setting now $r=\bar F_2^{-1}\left (\bar F_1(s)\right )$, we obtain
$$\bar F_1^{-1}\left (\bar F_2(t+r)\right)=\bar F_1^{-1}\left (\bar F_2(r)\right )+\bar F_1^{-1}\left (\bar F_2(t)\right),\quad r,t\geq 0,$$
from which
$$\bar F_1^{-1}\left (\bar F_2(z)\right)=az,\quad z\geq 0,$$
for some $a>0$.
By substituting in (\ref{e1}), we get
\begin{equation}\label{sol_strong}\bar F(s,y)=\bar F_1\left (s+ay\right),\quad s,y\geq 0\end{equation}
and from (\ref{MO_strong}), we recover that $d_{s,t}(x)=\frac{\bar F_1\left (s+at+\bar F_1^{-1}(x)\right )}{\bar F_1\left (s+at\right )}$. Moreover, since $\bar F$ is a bivariate survival distribution function, from (\ref{sol_strong}), $\bar F_1$ is necessarily convex.  \end{proof}
Notice that from (\ref{time_strong}) the couples $(s_1,t_1)$ and $(s_2,t_2)$ for which $s_1+at_1=s_2+at_2$ are associated to the same distortion.
\
\begin{remark}\label{rem_strong}
The class of functions satisfying (\ref{MO_strong}) coincides with the class of functions satisfying the functional equation (\ref{Strong Pseudo Equation}) for a suitable generator.
In fact, setting $h(x)=\bar H\left (-\ln x)\right )$ in 
(\ref{strong_sol}), as shown in Ricci (2024), we get that  
\begin{equation*}\label{eq_strong}
\bar F(x+s,y+t)=h\left (e^{-x-s-ay-at}\right)=\bar F(s,t)\otimes_h\bar F(x,y)
\end{equation*}
meaning that (\ref{MO_strong}) is equivalent to (\ref{Strong Pseudo Equation}). As mentioned in Section \ref{prerequisites}, survival distributions of this type have already been extensively studied in literature and we will not further analyze them.

\end{remark}

\subsection{The generalized weak bivariate Marshall-Olkin functional equation}\label{MO-W}
Let $(X,Y)$ be distributed according to the joint survival distribution function $\bar F$ and let $\bar F_t$ denote the joint survival distribution function of ${\bf X}_t$. The generalized version of the Marshall and Olkin (2015) functional equation (\ref{weak_functional}) is 
\begin{equation}\label{MO_weak}\bar F_t(x,y)=\frac{\bar F(x+t,y+t)}{\bar F(t,t)}=d_{t}\left (\bar F(x,y)\right ),\quad t,x,y \geq 0,\end{equation}
where, for all $t\geq 0$, $d_{t}$ is a strictly increasing bijection of $[0,1]$ with $d_{0}(x)=x$.
\vskip 0.3cm
The proof of the following Proposition follows the same ideas contained in the proof of Proposition 4.1 in Marshall and Olkin (2015).
\begin{proposition}\label{Prop_weak}
A bivariate survival function $\bar F$ satisfies (\ref{MO_weak}) if and only if there exists a strictly increasing bijection $h$ of $[0,1]$ for which
\begin{equation}\label{pwlmp}\bar F(x+t,y+t)=\bar F(x,y)\otimes_h\bar F(t,t).\end{equation}
In this case
\begin{equation}\label{d}d_t(x)=\frac{h\left (e^{-t}h^{-1}(x)\right )}{h\left (e^{- t}\right )}=\frac{\bar H\left(t+\bar H^{-1}(x)\right )}{\bar H(t)},\end{equation}
where $\bar H(x)=\bar F(x,x)=h\left(e^{-x}\right)$.
\end{proposition}

\begin{proof}
If $\bar F$ satisfies (\ref{pwlmp}) with respect to a given generator $h$, then it can be easily verified that (\ref{MO_weak}) holds true with 
$d_t(x)=\frac{h\left (h^{-1}\left (\bar F(t,t)\right )h^{-1}(x)\right )}{\bar F(t,t)}$.
Moreover, since the solutions of (\ref{pwlmp}) are given in (\ref{PWBLMPh_eq}), $\bar F(t,t)=h\left (e^{-\lambda t}\right)$. Taking into account that the generator $h$ is defined up to the composition with a power function (see Section \ref{prerequisites}),
we can choose $h$ such that $\bar F(x,x)=h\left(e^{-x}\right)$ and (\ref{MO_weak}) holds true with $d_t$ given by (\ref{d}).
\vskip 0.1 cm

\noindent Let us now assume that (\ref{MO_weak}) holds true for a bivariate survival functiona $\bar F$. If $(X,Y)$ is a random vector with joint survival distribution $\bar F$ and $\bar H$ is the survival distribution of $W=\min(X,Y)$,
substituting $y=x$ in (\ref{MO_weak}), we get
$$\frac{\bar F(x+t,x+t)}{\bar F(t,t)}=d_{t}\left (\bar F(x,x)\right ),\quad t,x \geq 0,$$
that is equivalent to
$$\frac{\bar H(x+t)}{\bar H(t)}=d_{t}\left (\bar H(x)\right ),\quad t,x \geq 0.$$
Then (see Subsection \ref{Kaminsky1}) the distortion $d_t$ is uniquely given by 
$d_t(x)=\frac{\bar H\left (t+\bar H^{-1}(x)\right )}{\bar H(t)}$ and, setting $h(x)=\bar H(-\ln x)$ and substituting in 
(\ref{MO_weak}), we have that 
$$\bar F(x+t,y+t)=h\left (e^{-t}h^{-1}\left (\bar F(x,y)\right )\right)=h\left (h^{-1}\left (\bar F(t,t)\right)h^{-1}\left (\bar F(x,y)\right )\right),$$
so (\ref{pwlmp}) is satisfied.
\end{proof}

\begin{remark}\label{conditional_PLMP} By Proposition \ref{Prop_weak}, $\bar F$ satisfies (\ref{MO_weak}) if and only if it satisfies the functional equation (\ref{Weak Pseudo Equation}) with respect to the distortion $h$. But since this class coincides with the class of survival distributions obtained by distorting through $h$ a function $\bar G$ satisfying the classical bivariate weak lack-of-memory property functional equation (\ref{weak standard equation}) (see (\ref{PWBLMP_eq})), we have that 
\begin{equation*}\label{MO_weak_distorted}\bar F_t(x,y)=h_{t}\left (\bar G(x,y)\right ),\quad t,x,y \geq 0, \end{equation*}
where 
\begin{equation}\label{t_dep_dist}h_t(x)=\frac{h\left (e^{-t}x\right )}{h\left (e^{- t}\right )}\end{equation} 
is a strictly increasing bijection of the interval $[0,1]$ for every $t\geq 0$.
\end{remark}
\vskip 0.2cm
Notice that, since the pseudo product in (\ref{pwlmp}) is determined by $h$ up to powers of its argument, in the statement of Proposition \ref{Prop_weak}
we have selected a specific $h$ that is understood in the rest of the paper.

\subsubsection{Bivariate distributions satisfying the generalized weak Marshall-Olkin functional equation: mixing approach}\label{mixing_sec}

In this Subsection we will construct a family of distributions satisfying (\ref{MO_weak}) starting from a specific family of distributions satisfying the bivariate weak lack-of-memory property (\ref{weak standard equation}) and  through a mixing approach.
\vskip 0.2cm
\noindent In Theorem 5.1 in Mulinacci (2018) it is shown that the survival distributions family
\begin{equation}\label{MU_LMP}
\begin{aligned}\bar G_{\alpha,\gamma,\alpha_1,\alpha_2}(x,y)&=\left (\alpha_2 e^{\gamma x}+\alpha_1 e^{\gamma y}+\left(1-\alpha_1-\alpha_2\right)e^{\gamma\,\max (x,y)}\right )^{-\frac{1}{\alpha}}=\\
&=\left\{
\begin{array}{cc}
e^{- \frac{\gamma y}{\alpha}}\left (\alpha _1+(1-\alpha_1)e^{\gamma(x-y)}\right )^{-\frac{1}{\alpha}},& x\geq y\geq 0\\
e^{- \frac{\gamma x}{\alpha}}\left (\alpha _2+(1-\alpha_2)e^{\gamma(y-x)}\right )^{-\frac{1}{\alpha}},& 0\leq x< y\\
\end{array}
\right .,\end{aligned}\end{equation}
with $\alpha_1,\alpha_2\in (0,1)$, $\alpha_1+\alpha_2\leq 1$ and $\alpha,\gamma >0$, satisfies the standard weak lack-of-memory property in (\ref{weak standard equation}). 
The marginal distributions are $\bar G_{\alpha,\gamma,\alpha_i,i}(z)=\left (\alpha_i+(1-\alpha_i)e^{\gamma z}\right )^{-\frac{1}{\alpha}}, \ i = 1,2$, with hazard rates $r_{\alpha ,\gamma,\alpha_i,i}(z)= \frac{\gamma (1-\alpha_i)}{\alpha} \frac{e^{\gamma z}}{\alpha_i+(1-\alpha_i)e^{\gamma z}}, \ i = 1,2.$
\vskip 0.2cm
\noindent Let $Z$ be a strictly positive random variable and $(X,Y)$ be a random vector for which
$$\mathbb P\left (X>x,Y>y\vert Z=z\right)=\bar G^{z}_{\alpha,\gamma,\alpha_1,\alpha_2}(x,y);
$$
this means that $Z$ is a common multiplicative stochastic factor affecting $r_{\alpha,i}$, for $i=1,2$ through the parameter $\alpha$.
\vskip 0.2cm
\noindent So the joint survival distribution of $(X,Y)$ is 
\begin{equation}\label{mix_laplace}\begin{aligned}\bar F(x,y)& =\mathbb E\left [\mathbb P\left (X>x,Y>y\vert Z\right)\right ]
=\mathbb E\left [\bar G_{\alpha,\gamma,\alpha_1,\alpha_2}^Z(x,y)\right ]=\\
&=  M_Z\left (\ln \left(\bar G_{\alpha,\gamma,\alpha_1,\alpha_2}(x,y)\right )\right )=h\left(\bar G_{\gamma,\gamma,\alpha_1,\alpha_2}(x,y)\right ),
\end{aligned}\end{equation}
where $M_Z(u)=\mathbb E\left [e^{uZ}\right ]$ is the moment generating function of the random variable $Z$ and $h(z)=M_Z\left (\frac \gamma\alpha\ln z\right )$. Then $\bar F$ satisfies (\ref{MO_weak}) with 
$$d_t(z) = \frac{M_{Z}\left(M^{-1}_Z(z)-\frac \gamma\alpha t\right)}{M_Z\left (-\frac \gamma\alpha t\right )}$$
and 
$$\bar F_t(x,y)=\frac{M_Z\left (-\frac 1\alpha\ln\left (\alpha_2 e^{\gamma x}+\alpha_1 e^{\gamma y}+\left(1-\alpha_1-\alpha_2\right)e^{\gamma \,\max (x,y)}\right )-\frac \gamma\alpha t\right)}{M_Z\left (-\frac \gamma\alpha t\right )}.$$
Here below we will consider some specific examples in which the moment generating function of the mixing variable $Z$ is known in closed form.

\begin{example}\label{ex_mix}
\begin{enumerate}
\item $Z$ gamma distributed, with parameters $a>0$ and $1$: $M_Z(u)=\left (1-u\right )^{-a}$, $h(z)=\left (1-\frac \gamma\alpha\ln(z)\right )^{-a}$, $z\in (0,1]$, $d_t(z) = \left(\frac{\frac \gamma\alpha t + z^{-\frac{1}{a}}}{1+\frac \gamma\alpha t}\right)^{-a}, \ z \in (0,1]$, and 
$$\bar F_t(x,y)=\left (1+\frac 1{\alpha+\gamma t}\ln\left (\alpha_2 e^{\gamma x}+\alpha_1 e^{\gamma y}+\left(1-\alpha_1-\alpha_2\right)e^{\gamma\,\max (x,y)}\right )\right)^{-a}$$
with $\alpha_1,\alpha_2\in (0,1)$ and $a,\gamma, \alpha>0$.
\item $Z$ stable distributed, with parameter $a\in (0,1]$: $M_Z(u)=e^{-\vert u\vert^a}$, $h(z)=e^{-\left (\frac \gamma\alpha\ln\left(\frac 1z\right )\right )^a}$, $z\in (0,1]$, $d_t(z) = \frac{e^{-\left(\frac \gamma\alpha t + (-\log(z))^{\frac{1}{a}}\right)^a}}{e^{-\left (\frac \gamma\alpha t\right)^a}}, \ z \in (0,1]$, and 
$$\bar F_t(x,y)=\exp\left\{-\left (\frac \gamma\alpha t+ 
\frac{1}{\alpha}\ln\left (\alpha_2 e^{\gamma x}+\alpha_1 e^{\gamma y}+\left(1-\alpha_1-\alpha_2\right)e^{\gamma\,\max (x,y)}\right )\right )^a+\left (\frac \gamma\alpha\right )^a t^a\right\}$$
with $\alpha_1,\alpha_2\in (0,1)$, $a\in (0,1]$ and $\alpha, \gamma>0$.
\item $Z$ Sibuya distributed, with parameter $a\in (0,1]$: $M_Z(u)=1-(1-e^
u)^a$, $h(z)=1-\left (1-z^{\frac \gamma\alpha}\right )^a$, $z\in [0,1]$,
$d_t(z) = \frac{1-\left(1-e^{-\frac \gamma\alpha t} + e^{- \frac \gamma\alpha t} (1-z)^{\frac{1}{a}}\right)^a}{1-\left (1-e^{-\frac \gamma\alpha t}\right)^a}, \ z \in [0,1]$, and
$$\bar F_t(x,y)=\frac{1-\left (1-e^{-\frac \gamma\alpha t}\left(\alpha_2 e^{\gamma x}+\alpha_1 e^{\gamma y}+\left(1-\alpha_1-\alpha_2\right)e^{\gamma\,\max (x,y)}\right )^{-\frac{1}{\alpha}}\right)^a}{1-\left(1-e^{-\frac \gamma\alpha t}\right)^a}
$$
with $\alpha_1,\alpha_2\in (0,1)$, $a\in (0,1]$ and $\alpha, \gamma>0$.
\item $Z$ distributed according to the logarithmic series distribution with parameter $a >0$:
$M_Z(u)=-\frac 1a\ln\left (1+(e^{-a}-1)e^u\right )$.
This framework can be re-parametrized by setting $\theta=e^{-a}-1\in (-1,0)$: in this case, we have that $M_Z(u)=\frac{\ln\left (1+\theta e^u\right )}{\ln (\theta +1)}$, $h(z)=\frac{\ln\left (1+\theta z^{\frac \gamma\alpha}\right )}{\ln (\theta +1)}$, $z\in[0,1]$, $d_t(z) = \frac{\log\left (1+e^{-\frac \gamma\alpha t} ((\theta+1)^z-1)\right)}{\log\left(1+\theta e^{- \frac \gamma\alpha t}\right)}, \ z \in [0,1]$, and
$$\bar F_t(x,y)=\frac{\ln\left (1+\theta e^{-\frac \gamma\alpha t}\left (\alpha_2 e^{\gamma x}+\alpha_1 e^{\gamma y}+\left(1-\alpha_1-\alpha_2\right)e^{\gamma\,\max (x,y)}\right )^{-\frac{1}{\alpha}}\right )}{\ln\left(1+\theta e^{-\frac \gamma\alpha t}\right)}
$$
with $\alpha_1,\alpha_2\in (0,1)$, $\theta\in (-1,0)$ and $\alpha, \gamma>0$.
\end{enumerate}
\end{example}

\section{Aging and generators}\label{Aging and generators}
The absence of aging effects characterizes the lack-of-memory property at any dimensional level, both in the strong and in the weak cases.
The importance of aging effects is well known in reliability theory as well as in actuarial risk
both in the life as well in the non-life cases.

Here, we focus on the two well known notions of increasing (decreasing) failure rate, IFR (DFR), and new better (worse) than used, NBU (NWU), that in the unidimensional case are defined by (see, for example, among the wide literature, Marshall and Olkin, 2007):
\begin{itemize}
\item IFR (DFR): $\bar F_{X_s}(x)\geq (\leq)\bar F_{X_t}(x)$, for all $x\geq 0$, for all $t\geq s\geq 0$, 
\item NBU (NWU): $\bar F(x)\geq (\leq)\bar F_{X_t}(x)$, for all $x\geq 0$, for all $t\geq 0$, 
\end{itemize}
where $X_t$ is the excess of $X$ above $t$ defined in (\ref{survival_one}).

It is well known that IFR (DFR) is equivalent to the log-concavity (log-convexity) of $\bar F$ (see Section 4 in Marshall and Olkin, 2007). Obviously, if $\bar F$ is IFR (DFR), then it satisfies the NBU (NWU) property, while if $\bar F_{X_t}$ satisfies the NBU (NWU) property for all $t$ , then $\bar F$ is IFR (DFR) (see Section 5 in Marshall and Olkin, 2007). 

In (\ref{d_tuni}), as well as in Proposition \ref{Prop_weak}, $d_t$ is defined in terms of a univariate survival function $\bar H$ (or, equivalently, in terms of a distortion $h$ given by $h(t)=\bar H\left(-\log(t)\right )$), that is  $d_t(x)=\frac {\bar H\left(t+\bar H^{-1}(x)\right )}{\bar H(t)}$,
where $\bar H$ is a strictly decreasing survival function of a strictly positive random variable. A similar representation holds true for the distortions $d_{s,t}$ in Proposition \ref{strong_strong}, that is (see (\ref{time_strong})) $d_{s,t}(x)=\frac{\bar H(s+ta+\bar H^{-1}(x)}{\bar H(s+ta)}$ with $a>0$. Since $d_{s,t}=d_{s+ta}$, we will restrict our analysis to $d_t$: all properties and arguments that we will study and apply to $d_t$ clearly hold true also for $d_{s,t}$.
\smallskip

\noindent The NBU (NWU) property for $\bar H$ is obviously equivalent to 
\begin{equation}\label{sub_x}d_t(x)\leq (\geq) x,\quad \text{ for all }t\geq 0,\, x\in [0,1] \end{equation}
while the IFR (DFR) property for $\bar H$ is equivalent to 
$$d_s(x)\geq (\leq)d_t(x),\quad \text{for all }t\geq s\geq  0,\, x\in [0,1].$$

The aging properties satisfied by the underlying survival distribution function $\bar H$ are inherited by the bivariate survival functions satisfying the generalized weak Marshall-Olkin functional equation (\ref{MO_weak}). More specifically, if $\bar H$ satisfies the IFR (DFR) or the NBU (NWU) property, then the bivariate survival functions satisfying (\ref{MO_weak}), with the associated $d_t$, satisfy the corresponding bivariate versions:
\begin{itemize}
\item 2-IFR (2-DFR): $\bar F_s(x,y)=d_s\left (\bar F(x,y)\right)\geq (\leq)d_t\left (\bar F(x,y)\right)=\bar F_t(x,y)$, for all $x,y\geq 0$, for all $t\geq s\geq 0$, 
\item 2-NBU (2-NWU): $\bar F(x,y)\geq (\leq)d_t\left (\bar F(x,y)\right)=\bar F_t(x,y)$, for all $x,y\geq 0$, for all $t\geq 0$.
\end{itemize}

We provide here some specific examples.
\begin{example}\label{Many_distortions}
\begin{enumerate}
\item $\bar H(x)=e^{-(ax)^\alpha}$, $a,\alpha>0$ (Weibull distribution): $h(t)=e^{-(-a\log(t))^\alpha}$ and $d_t(u)=e^{(at)^\alpha-\left( \left(-\log\, u\right)^{\frac 1\alpha}+at \right )^\alpha}$, which is decreasing in $t$ when $\alpha>1$ and increasing in $t$ when $0<\alpha<1$.
\item $\bar H(x)=e^{-\xi(e^{\mu x}-1)}$, $\xi,\,\mu>0$ (Gompertz distribution): $h(t)=e^{-\xi\left (t^{-\mu}-1\right )}$ and 
$d_t(u)=u^{e^{\mu\, t}}$ that is decreasing in $t$.
\item $\bar H(x)=(ax+1)^{-\frac 1\mu}$, $a,\mu>0$ (Pareto distribution): $h(t)=\left (1-a\log(t)\right)^{-\frac 1\mu}$ and 
$d_t(u)=\left (\frac{at+1}{u^{-\mu}+at}\right )^{\frac 1\mu}$, which is increasing in $t$.
\item $\bar H(x)=\left (\theta e^{ax}+1-\theta\right )^{-1}$, $a,\theta>0$: $h(t)=\left (\theta t^{-a}+1-\theta\right )^{-1}$ and 
$d_t(u)=\frac{\theta e^{at}+1-\theta}{e^{at}\left (u^{-1}-1+\theta\right)+1-\theta}$, which is decreasing in $t$ when $0<\theta <1$ and increasing in $t$ when $\theta>1$.
\item $\bar H(x)=\frac{\log(\theta e^{-ax}+1)}{\log(\theta +1)}$, $a>0$, $\theta > -1$: $h(t)=\frac{\log(\theta t^a+1)}{\log(\theta +1)}$ and 
$d_t(u)=\frac{\log\left (e^{-at}\left ((\theta +1)^u-1\right )+1\right )}{\log\left (\theta e^{-at}+1\right)}$, which is decreasing in $t$ when $\theta >0$ and increasing in $t$ when $-1<\theta<0$.
\item $\bar H(x)=\frac 4\pi\arctan\left(e^{-ax}\right)$, $a>0$: $h(t)=\frac 4\pi\arctan\left(t^a\right)$ and $d_t(u)=\frac{\arctan\left(e^{-at}\tan\left(\frac \pi 4u\right )\right)}{\arctan\left (e^{-at}\right )}$, which is always decreasing in $t$. 
\end{enumerate}\end{example}
\begin{remark}
In risk analysis the mean-excess function of a positive and integrable random variable $X$ is an important tool to study the tail behavior of the associated distribution, that is to analyze the tail riskiness associated to $X$. In the bivariate case, that is considering two risks $X_1$ and $X_2$, it can be useful to consider the average excess of each of the two with respect to a given threshold, given that both exceed it. In the case of distributions satisfying the weak bivariate Marshall-Olkin functional equation, we have that
$$e_{X_i}(t)=\mathbb E\left [X_i-t\vert X_1>t,X_2>t\right ]=\int_0^{+\infty}\bar F_{X_i-t\vert X_1>t,X_2>t}(z)dz=\int_0^{+\infty}d_t\left(\bar F_{X_i}(z)\right )dz,\quad i=1,2,$$
and the behaviour of $e_{X_i}$, for $t$ large, depends on the distortion $d_t$ which is determined by the distribution of $\min(X_1,X_2)$.

In fact, let $d_\infty(x)=\underset{t\to \infty}\lim d_t(x)$ exist for every $x\in(0,1)$.
If the distribution of $\min(X_1,X_2)$ satisfies the NBU property for all $t\geq 0$, then, since $d_t(x)\leq x$ for all $t\geq 0$ and $x\in [0,1]$ (see (\ref{sub_x})), $\underset{t\to \infty}\lim e_{X_i}(t)=\int_0^{+\infty}d_\infty\left(\bar F_{X_i}(z)\right )dz<+\infty$.
If the distribution of $\min(X_1,X_2)$ satisfies the DFR property then, since $d_t(x)$ is increasing with respect to $t\geq 0$ for all $x\in [0,1]$, again   
$\underset{t\to \infty}\lim e_{X_i}(t)=\int_0^{+\infty}d_\infty\left(\bar F_{X_i}(z)\right )dz$ but the right-hand side can be $+\infty$.
In particular, if the survival distribution of $\min(X_1,X_2)$ is heavy tailed, that is  $\underset{t\to \infty}\lim e_{\min(X_1,X_2)}(t)=+\infty$, then also $\underset{t\to \infty}\lim e_{X_i}(t)=+\infty$, for $i=1,2$ (see cases 1. and 3. in Example \ref{Many_distortions}).
\end{remark}

\noindent It can be easily verified that
(\ref{sub_x}) is equivalent to
 \begin{equation*}\label{sub_mult}h(x  y ) \leq  (\geq)\ h(x) \ h(y),\quad \text{ for all } x, y \in [0,1],\end{equation*}
that is to the fact that the associated generator $h$ is sub(super)-multiplicative.
\vskip 0.2cm
A sufficient condition for the sub-multiplicativity of a generator $h$ is given by the following Proposition:
    \begin{proposition}
        Let $h:[0,1] \rightarrow [0,1]$ be a strictly increasing and concave bijection such that $h^{'''}(x) \leq 0$. Then $h$ is sub-multiplicative in $[0,1]$.
    \label{sub_prop}
    \end{proposition}
    \begin{proof}
        Let us define $g:[0,1] \times [0,1]$ as
        $$g(u,v) = h(uv)-h(u) h(v),$$ then $h$ is sub-multiplicative in $[0,1]$ if and only if $g$ is non-positive.
        On the sides of the square $[0,1] \times [0,1]$,
        $g$ is equal to $0$, so it is sufficient to prove that there are not maximum points inside the square.
        But
        \begin{equation*}
            \frac{\partial^2 g (u,v) }{\partial^2 u} = 
            v^2 h^{''}(uv)-h(v) h^{''}(u)  \geq 
            h(v) [h^{''}(uv)-h^{''}(u)] \geq 0,
        \end{equation*}
       thanks to the concavity of $h$, the decreasingness of $h^{''}$ and noticing that a strictly increasing and concave function lies above the bisector of the first quadrant in the interval $[0,1]$.
        Since $\frac{\partial^2 g}{\partial^2 u}$ is non-negative, there are no maximum points inside the square, meaning that $g(u,v) \leq 0 \ \forall (u,v) \in [0,1] \times [0,1].$
    \end{proof}
An analogous result for super-multiplicativity holds true and can be proved similarly:
    \begin{proposition}
        Let $h:[0,1] \rightarrow [0,1]$ be a strictly increasing and convex bijection such that $h^{'''}(x) \geq 0$ and $h(x) \geq x^2$, $\forall x \in [0,1]$. Then $h$ is super-multiplicative in $[0,1]$.
    \label{super_prop}
    \end{proposition}
\begin{example} By applying Propositions \ref{sub_prop} and \ref{super_prop}, it can be  verified that:
\begin{itemize}
\item $h(x) = \frac{3x-x^3}{2}$ is a sub-multiplicative generator,
\item $h(x) = \frac{\sin(\theta x)}{\sin(\theta)}$, for $0 < \theta < \frac{\pi}{2}$, is a sub-multiplicative generator,
\item $h(x) = \frac{1}{4} x^3 + \frac{1}{2} x^2 + \frac{1}{4} x$ is a super-multiplicative generator.
\end{itemize}
\end{example}

\section{Dependence structure dynamics in the Bivariate Weak Case}\label{Dependence}
Aim of this section is to analyze the dynamics of the dependence structure of ${\bf X}_{t}$ (see (\ref{survival_two_weak})) induced by the generalized weak bivariate Marshall-Olkin functional equation (\ref{MO_weak}). 
Since, thanks to Remark \ref{conditional_PLMP}, the conditional distribution $\bar F_t$ is obtained through a time dependent distortion of a function satisfying the classical bivariate weak lack-of-memory property functional equation (\ref{weak standard equation}), some useful properties and formulas provided in 
Ricci (2024) for $\bar F=\bar F_0$ still hold true for $\bar F_t$.

In particular, from (\ref{sing}), if $h$ is twice differentiable with $h^\prime(x)>0$, for all $x\in [0,1]$,  we can immediately conclude that the singularity mass of the distribution $\bar F_t$ is independent of $t$ being independent of the distortion but only dependent on $\bar G=h^{-1}\left (\bar F\right)$ satisfying the functional equation (\ref{weak standard equation}). However, the distribution of the singularity mass of $\bar F_t$ on the straight line $x=y$ changes with $t$ according to $S_t(x)=\mathbb P(X=Y)\cdot h_t(e^{-x})$.
\medskip

By (\ref{MO_weak}), if $C$ is the survival copula associated to $\bar F$, then the copula function associated to $\bar F_t$ is
\begin{equation}\label{t_cop}C_t(u,v)=d_t\left (C\left (d_t^{-1}(u),d_t^{-1}(v)\right )\right ),\end{equation}
and the dependence structure evolves with time according to the above distorted copula. Thanks to Remark \ref{conditional_PLMP}, 
if $\bar G=h^{-1}\left (\bar F\right )$ and $\bar G_1$ and $\bar G_2$ are the corresponding marginal survival distribution functions (see (\ref{PWBLMP_eq}) and 
(\ref{PWBLMPh_eq})),
the time dependent copula $C_t$ in (\ref{t_cop}) can be alternatively written as
\begin{equation}\label{c_cop}C_t(u,v)=h_t\left (C_{\bar G}\left (h_t^{-1}(u),h_t^{-1}(v)\right )\right )\end{equation}
where 
\begin{equation}\label{copula_G}C_{\bar G}(w_1,w_2)=\bar G\left (\bar G_1^{-1}(w_1),\bar G_2^{-1}(w_2)\right ),\end{equation}
with $h_t$ given by (\ref{t_dep_dist}).

\subsection{Kendall's function}\label{PWB_Kendall}

The Kendall's function of a random vector $(X,Y)$ with cumulative distribution $H$ is defined as
$$K(t) = P(H(X,Y) \leq t), \ t \in [0,1].$$
Since it actually only depends on the copula associated to $H$, it turns out to be a very useful tool to study the dependence between the components of a bivariate random vector (see among the others, Nelsen et al., 2003, Nelsen, 2006, and Joe, 2014).
In the case of perfect positive dependence, $K(t) = t$, $ t \in [0,1]$, while, in case of independence, $K(t) = t-t\log(t)$, $t \in (0,1]$.
Moreover, it is well known that the Kendall's $\tau$ is a statistics used to measure the ordinal association between two random variables and that it can be recovered from the Kendall's function through 
$\tau = 3-4\int_{0}^{1} K(t) \ dt$.
\vskip 0.1 cm
In Ricci (2024), a general formula for the Kendall's function is provided in case of a bivariate distribution satisfying (\ref{Weak Pseudo Equation}), that is of type (\ref{PWBLMPh_eq}). According to the notation used in (\ref{PWBLMPh_eq}), $K$ is given by
\begin{equation*}
K(s)=
s-H(h^{-1}(s))
\label{Kendall_equation}
\end{equation*}
where
$$H(v)=h^\prime\left (v\right)v\left[2\ln\left(v\right)+\frac 1\lambda \left(J_1(v)+J_2(v)\right)\right]\text{ and }
J_i(v)=\int_0^{\bar G_i^{-1}\left (v\right )}\frac{g_i^2(z)}{\bar G_i^2(z)}dz,\,i=1,2,$$
with $g_i=-\bar G_i^\prime$ for $i=1,2$.
\vskip 0.2cm

As a consequence of Remark \ref{conditional_PLMP}, the expression of the Kendall's function $K_t$ of the copula $C_t$ in (\ref{c_cop}) is given by
$$K_{t}(s)=
s-H_{h_t}\left (h_t^{-1}(s)\right ) = s-H_{h_t}\left(\frac{h^{-1}(s \ h(e^{-t}))}{e^{-t}}\right),$$
where
\begin{equation*}
H_{h_t}(v)=h_t^\prime\left (v\right)v\left[2\ln\left(v\right)+\frac 1\lambda \left(J_1(v)+J_2(v)\right)\right] =\frac{e^{-t}h^{'}(e^{-t} v)}{h(e^{-t})} v \left[2\ln\left(v\right)+\frac 1\lambda \left(J_1(v)+J_2(v)\right)\right],
\end{equation*}
with $J_i$, $i=1,2$, obtained considering 
$\bar G_i=h^{-1}\left (\bar F_i\right )$ where $\bar F_i$, $i=1,2$, are the $i$-th marginal survival distribution of $\bar F$ satisfying (\ref{MO_weak}).

\begin{example}\label{Example6}
Let us consider the setup of Subsection \ref{mixing_sec}. It can be verified that, for the survival distribution family (\ref{MU_LMP}),
\begin{equation*}
    J_i(x) = \frac{\gamma}{\alpha^2} \left(\alpha_i x^\alpha - \alpha \log(x)-\alpha_i\right).
\end{equation*}
\noindent The expression of the Kendall's functions in all cases analized in Example \ref{ex_mix} can be easily determined.
\begin{enumerate}
\item
If $h(x) = (1-\frac\gamma\alpha\ln(x))^{-a}, \ a,\alpha,\gamma > 0$, then 
$$K_t(s) = 
s - \frac{a \gamma (\alpha_1 + \alpha_2)}{\alpha^2} \cdot s^{\frac{a+1}{a}} \cdot \left(1 + \frac{\gamma}{\alpha} t\right)^{-1} \cdot \left( \exp\left( t \alpha + \frac{\alpha^2}{\gamma} \left(1 - s^{-1/a} \left(1 + \frac{\gamma}{\alpha} t \right) \right) \right) - 1 \right).
$$
\item
If $h(x) = e^{-(-\frac{\gamma}{\alpha}\log(x))^a}, a\in (0,1]$, then 
\begin{equation*}
K_{t}(s) = s \left(1 - \frac{a\gamma}{{\alpha}^{2}}\left({\alpha}_{2} + {\alpha}_{1}\right)  \left(\left(\frac{t{\gamma}}{\alpha}\right)^{a} - \ln\left(s\right)\right)^{\frac{a - 1}{a}} \left(\mathrm{e}^{t{\alpha} - \frac{{\alpha}^{2} \left(\left(\frac{t{\gamma}}{\alpha}\right)^{a} - \ln\left(s\right)\right)^{\frac{1}{a}}}{\gamma}} - 1\right)\right).
\end{equation*}
\item
If $h(x) = 1-\left (1-x^{\frac{\gamma}{\alpha}}\right)^a, \ a \in ( 0,1]$, then 
\begin{equation*}
K_t(s)  =  s - \frac{a \left({\alpha}_{2} + {\alpha}_{1}\right) {\gamma}}{{\alpha}^{2} v_t} \left(\mathrm{e}^{t{\alpha}} \left(1 - \left(1 - s v_t\right)^{\frac{1}{a}}\right)^{\frac{{\alpha}^{2}}{\gamma}} - 1\right) 
          \left(1 - \left(1 - s v_t\right)^{\frac{1}{a}}\right) \left(1 - sv_t\right)^{\frac{a - 1}{a}}
\end{equation*}
where $v_t=1-\left (1-e^{-\frac {t\gamma}\alpha}\right )^a$.
\item
If $h(x) = \frac{\log\left(1+\theta x^{\frac{\gamma}{\alpha}}\right)}{\log(1+\theta)}, \ \theta \in (-1,0)$, then 
\begin{equation*}
 K_{t}(s) = s - \frac{{\gamma}\left({\alpha}_{2} + {\alpha}_{1}\right)  \left(\left( {\theta}\mathrm{e}^{-\frac{t{\gamma}}{\alpha}} + 1\right)^{s} - 1\right)}{{\alpha}^{2} \left( {\theta}\mathrm{e}^{-\frac{t{\gamma}}{\alpha}} + 1\right)^{s} \ln\left({\theta}\mathrm{e}^{-\frac{t{\gamma}}{\alpha}}  + 1\right)} \left(\theta^{-\frac{\alpha^2}\gamma}\mathrm{e}^{t{\alpha}}  \left(\left({\theta}\mathrm{e}^{-\frac{t{\gamma}}{\alpha}}  + 1\right)^{s} - 1\right)^{\frac{{\alpha}^{2}}{\gamma}} - 1\right).
\end{equation*}
\end{enumerate}
\end{example}

Other families of distributions satisfying the generalized bivariate weak Marshall-Olkin functional equation (\ref{MO_weak}) can be constructed starting from alternative underlying bivariate survival distributions satisfying the weak bivariate lack-of-memory property. Here, we analyze the case that allows to recover the bivariate Gompertz distribution introduced and studied in Marshall and Olkin (2015).
\vskip 0.2cm
This is the case of a generalization of the survival distribution function (\ref{MU_LMP}), still satisfying the bivariate weak lack-of-memory property (\ref{weak standard equation}), defined as 
\begin{equation*}\label{MU_LMP_G}
\bar G_{\alpha,\gamma_1,\gamma_2,\alpha_1,\alpha_2}(x,y)=\left\{
\begin{array}{cc}
e^{- \lambda y}\left (\alpha _1+(1-\alpha_1)e^{\gamma_1(x-y)}\right )^{-\frac{1}{\alpha}},& x\geq y\geq 0\\
e^{- \lambda x}\left (\alpha _2+(1-\alpha_2)e^{\gamma_2(y-x)}\right )^{-\frac{1}{\alpha}},& 0\leq x< y\\
\end{array}
\right .,\end{equation*}
with $\alpha_1,\alpha_2\in (0,1)$ and $\alpha,\gamma_1,\gamma_2 >0$. This is a bivariate survival distribution function if and only if 
$\frac 1\alpha\max\left(\gamma_1,\gamma_2\right )\leq \lambda\leq \frac 1\alpha\left (\gamma_1(1-\alpha_1)+\gamma_2(1-\alpha_2)\right )$
with marginal survival distribution functions $\bar G_{\alpha,\gamma_i,\alpha_i,i}(z)=\left (\alpha_i+(1-\alpha_i)e^{\gamma_i z}\right )^{-\frac{1}{\alpha}}, \ i = 1,2$:
the distribution (\ref{MU_LMP}) is recovered when $\gamma_1=\gamma_2=\gamma$ and $\lambda=\frac\gamma\alpha$. Moreover, the singularity mass on the line $x=y$ is $\frac{(1-\alpha_1)\gamma _1+(1-\alpha_2)\gamma_2}{\alpha\lambda}-1$.

%while the associated copula function is
%$$C_{\bar G_{\alpha,\gamma_1,\gamma_2,\alpha_1,\alpha_2}}(u,v)=\left\{\begin{array}{cc}
%\left (\frac{v^{-\alpha}-\alpha_2}{1-\alpha_2}\right )^{-\frac{\lambda}{\gamma_2}}
%\left (\alpha_1+(1-\alpha_1)\left(\frac{u^{-\alpha}-\alpha_1}{1-\alpha_1}\right )
%\left (\frac{1-\alpha_2}{v^{-\alpha}-\alpha_2}\right )^\frac{\gamma_1}{\gamma_2}
%\right )^{-\frac 1\alpha},&v\geq g(u)\\
%\left (\frac{u^{-\alpha}-\alpha_1}{1-\alpha_1}\right )^{-\frac{\lambda}{\gamma_1}}
%\left (\alpha_2+(1-\alpha_2)\left(\frac{v^{-\alpha}-\alpha_2}{1-\alpha_2}\right )
%\left (\frac{1-\alpha_1}{u^{-\alpha}-\alpha_1}\right )^\frac{\gamma_2}{\gamma_1}
%\right )^{-\frac 1\alpha},&v< g(u)\\
%\end{array}\right .
%$$
%where $g(u)=\left((1-\alpha_2)\left (\frac{u^{-\alpha}-\alpha_1}{1-\alpha_1}\right )^{\frac{\gamma_2}{\gamma_1}}+\alpha_2\right )^{-\frac 1\alpha}$

\begin{example}\label{MO15}
Let us consider the strictly increasing bijection of $[0,1]$, $h(x) = \exp(-\xi(x^{-1}-1))$, with $\xi>0$. The function
\begin{equation*}\begin{aligned}
 \bar{F}(x,y) &= h\left (\bar{G}_{\alpha,\gamma_1,\gamma_2,\alpha_1,\alpha_2}\right )=\\
&=     \begin{cases}
       \mathrm{e}^{-{\xi} \left(\mathrm{e}^{y{\lambda}} \left(\left(1 - {\alpha}_{1}\right) \mathrm{e}^{{\gamma_1} \left(x - y\right)} + {\alpha}_{1}\right)^{\frac{1}{\alpha}} - 1\right)} \ x \geq y \\
       \mathrm{e}^{- {\xi} \left(\mathrm{e}^{x{\lambda}} \left(\left(1 - {\alpha}_{2}\right) \mathrm{e}^{{\gamma_2} \left(y-x\right)} + {\alpha}_{2}\right)^{\frac{1}{\alpha}} - 1\right)} \ x < y \\
       \end{cases}\end{aligned}
    \label{F_gamma_1_gamma_2}
\end{equation*}
is a bivariate survival function if $\frac 1\alpha\max\left(\gamma_1,\gamma_2\right )\leq \lambda\leq \frac 1\alpha\left (\gamma_1(1-\alpha_1)+\gamma_2(1-\alpha_2)\right )$ and $\xi\geq 1$. 
\vskip 0.2cm

This class of distributions contains, as a particular specification, setting $\alpha = 1$, the survival distribution (\ref{MO2015}) studied in Marshall and Olkin (2015). In fact, (\ref{MO2015}) can be obtained setting $\lambda_i = \gamma_i$ and $\xi_i = \xi (1-\alpha_i)$, $i = 1,2$, with $\lambda \geq \max(\lambda_1,\lambda_2)$, $\lambda (\xi-1) \geq \max(\lambda_1 (\xi_1-1), \lambda_2 (1-\xi_2))$ and $\lambda_1 \xi_1 + \lambda_2 \xi_2 \geq \lambda \xi$. As a consequence, the analysis made in previous sections can be applied to the bivariate Gompertz distribution of Marshall and Olkin (2015). In particular, the survival distribution of the residual lifetimes is
\begin{equation*}\label{MO_app}
 \bar{F}_t(x,y) = 
     \begin{cases}
          e^{-e^{\lambda t}\left[\xi_1e^{\lambda y}\left (e^{\lambda_1(x-y)}-1\right )+\xi\left (e^{\lambda y}-1\right )\right]}  
 \ x \geq y \\
 e^{-e^{\lambda t}\left[\xi_2e^{\lambda y}\left (e^{\lambda_2(x-y)}-1\right )+\xi\left (e^{\lambda y}-1\right )\right]}   \ x < y \\
       \end{cases}.
\end{equation*}
and the time dependent associated Kendall Function is
\begin{equation*}
\begin{split}
& K_t(x) = x \left(1 -  {\xi}\mathrm{e}^{t{\lambda}} v_t(x) \left(\left(\frac{{\lambda}_{2} + {\lambda}_{1}}{\lambda} - 2\right) \ln\left(v_t(x)\right) +\frac 1\lambda\left(\frac{1}{v_t(x)} - 1\right)\sum_{i=1}^2{\lambda}_{i} \left(1 - \frac{{\xi}_{i}}{\xi}\right)
\right)\right)
\end{split}
\label{K_MO_2015}
\end{equation*}
where $v_t(x)=1-\frac 1\xi e^{-t\lambda}\ln(x)$.
\vskip 0.5cm
It can be verified that $K_t(x)\geq x-x\log(x)$ for all $x\in[0,1]$ and $t\geq0$, for any choice of admissible parameters: as a consequence, the bivariate Gompertz distribution allows to model joint residual lifetimes, with marginal Gompertz distributions, when they exhibit negative dependence.

\end{example}

\begin{example}Let us consider the strictly increasing bijection of $[0,1]$, $h(x) = \frac{\log(\rho x +1)}{\log(\rho +1)}, \rho \in (-1,0) \cup (0,\infty)$. The function
\begin{equation*}\label{Pippo}
\bar{F}(x,y) = h(\bar{G}_{1,\gamma_1,\gamma_2,\alpha_1,\alpha_2}(x,y))=
\begin{cases}
    \frac{\log(\rho e^{-\lambda y} \left((1-\alpha_1)e^{\gamma_1(x-y)}+\alpha_1\right)^{-1}+1)}{\log(\rho +1)} \\
    \frac{\log(\rho e^{-\lambda x} \left((1-\alpha_2)e^{\gamma_2(x-y)}+\alpha_2\right)^{-1}+1)}{\log(\rho +1)} \\
\end{cases}
\end{equation*}
is a survival distribution function if $\max(\gamma_1,\gamma_2) \leq \lambda \leq \gamma_1 (1-\alpha_1) + \gamma_2(1-\alpha_2)$ and $\rho \in (-1,0) \cup (0,\infty)$. 
The survival distribution of the residual lifetimes is
\begin{equation*}
\bar{F}_t(x,y) = h_t(\bar{G}_{1,\gamma_1,\gamma_2,\alpha_1,\alpha_2}(x,y))=
\begin{cases}
    \frac{\log(\rho e^{-\lambda t-\lambda y} \left((1-\alpha_1)e^{\gamma_1(x-y)}+\alpha_1\right)^{-1}+1)}{\log(\rho e^{-\lambda t}+1)} \\
    \frac{\log(\rho e^{-\lambda t-\lambda x} \left((1-\alpha_2)e^{\gamma_2(x-y)}+\alpha_2\right)^{-1}+1)}{\log(\rho e^{-\lambda t}+1)} \\
\end{cases}
\end{equation*}
 and the associated Kendall distribution function is
\begin{equation*}
\begin{split}
    K_t(x) =  x - &\frac{v_t(x) - 1}{\frac 1xv_t(x)\ln(v_t(x))}\cdot\\
&\cdot  \left(\frac{\left(\alpha_2 \gamma_{2} + \alpha_1 \gamma_{1}\right) \left(\mathrm{e}^{\lambda t}\left (v_t(x)-1\right ) - {\rho} \right)}{{\lambda}{\rho}} +
\left(2 - \frac{\gamma_{2} + \gamma_{1}}{\lambda}\right) \left(\ln\left(\frac{v_t(x) - 1}\rho\right)  + \lambda t\right)\right)
    \label{K82}
\end{split}
\end{equation*}
where $v_t(x)=\left (e^{-\lambda t}\rho+1\right )^x$.

Varying the admissible parameters, it can be shown that this distribution provides a very wide class of dependence structures. For example, Figure \ref{Fig1} displays a case in which the dependence is positive and increasing with the conditioning time $t$ (Left) and a case in which the dependence is negative and increasing with the time $t$ (Right). 

\begin{figure}[h!]
 \centering
 \includegraphics[width=.4\linewidth]{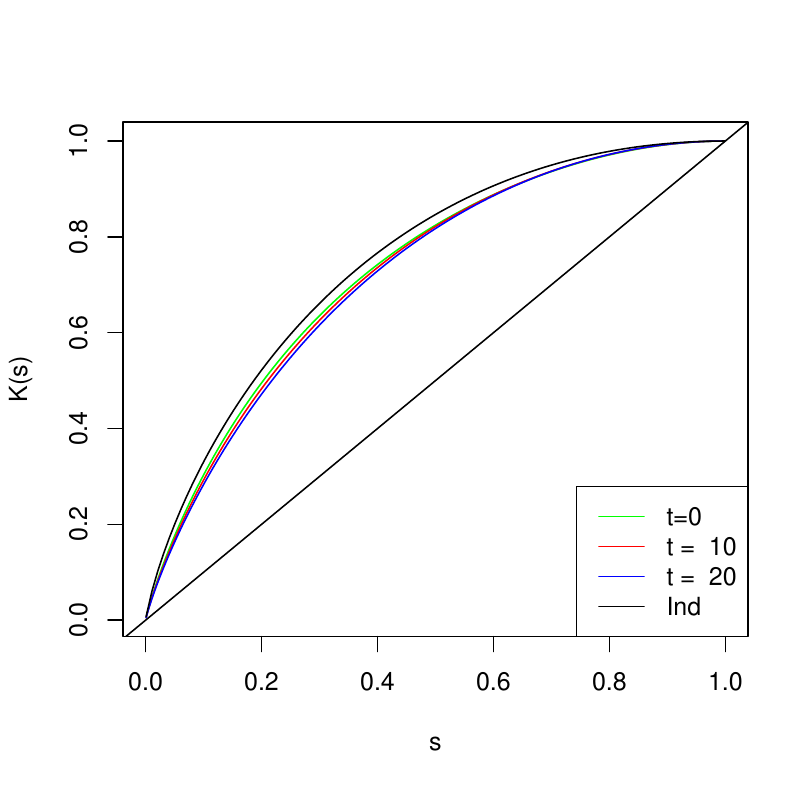}
\includegraphics[width=.4\linewidth]{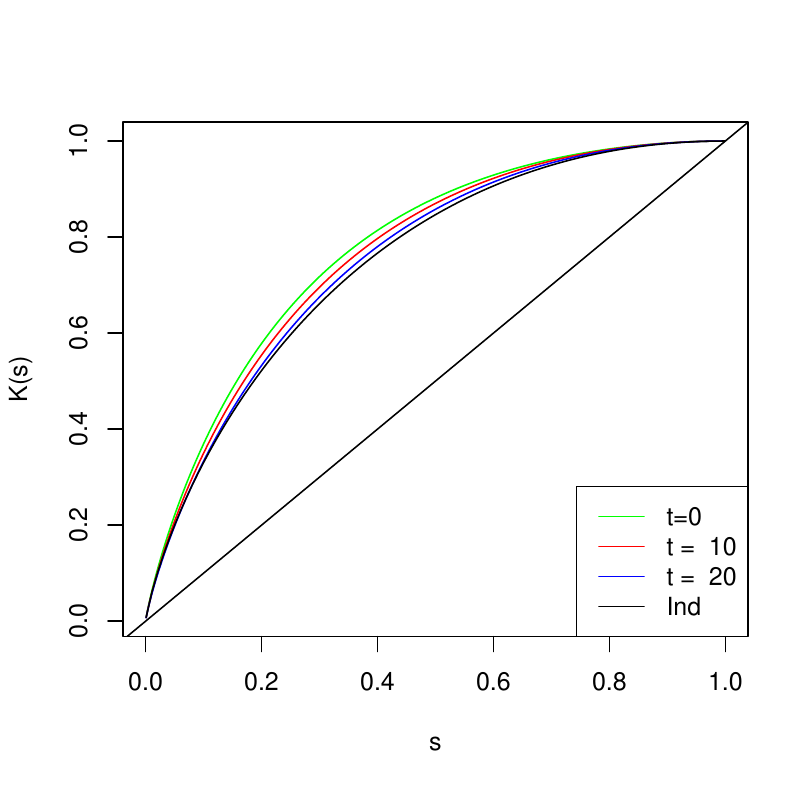}
   \caption{Left: $\lambda=0.0641$, $\alpha_1 = 0.31$, $\alpha_2 = 0.3611$, $\gamma_1 = 0.05$, $\gamma_2 = 0.0463$ and $\rho = 10$. Right: $\lambda=0.0726$, $\alpha_1 = 0.15$, $\alpha_2 = 0.2129$, $\gamma_1 = 0.046$, $\gamma_2 = 0.0426$ and $\rho = 10$.}
 \label{Fig1}
\end{figure}
Let us assume that these particular specifications are the distributions of the joint residual lifetimes of the two individuals in two different married couples where time $0$ is the starting observation time. In both cases, the probability of simultaneous death is equal to $0.01\%$ and parameters are chosen so that the average expectations of the two lifetimes at time $0$ are 39.5 and 43.4 years, respectively. The case of positive and increasing dependence over time is consistent with the idea that,
as a couple continues to survive together, their lives become mutually tied, in line with the well known broken hearth effect; the opposite case may happen if the death of one of the two can improve the life of the other one and so reduce his/her probability of death (see for example Gourieroux and Lu, 2015, for a discussion about this phenomenon). These two different scenarios clearly influence the premium of insurance policies written on the joint residual lifetimes (see for example Denuit et al., 2006, for a study on the influence of dependence on joint life insurance products). As an illustrative example, in Table \ref{Ann_Log}, we focus on a continuous joint whole life annuity (paying $1$ per year as long as both individuals are alive), where for the sake of simplicity we assume a discount factor equal to one, and we compare the net single premium for deferred and not deferred contracts with the case of independence between the two residual lifetimes.

\begin{table}
\centering
\begin{tabular}{|c|c|c|c|} 
\hline
& $t=0$ & $t=10$ &$t=20$\\
\hline
Fig. \ref{Fig1}, Left&27.2170 &18.4047 &11.8301\\
\hline
Indep.&25.7805 &16.9899 &10.5226\\
\hline
\end{tabular}
\begin{tabular}{|c|c|c|c|} 
\hline
& $t=0$ & $t=10$ &$t=20$\\
\hline
Fig. \ref{Fig1}, Right&24.0691& 15.4105& 9.2493\\
\hline
Indep.&25.2736 &16.6022 &10.3524\\
\hline
\end{tabular}
\caption{ Annuity net premiums in cases of Figure \ref{Fig1} versus independence}\label{Ann_Log}
\end{table}

\end{example}

\subsection{Tail dependence coefficients}
The lower and upper tail dependence coefficients of a copula $C$ are defined as
$$\lambda_L(C)=\underset{u\to 0^+}\lim \frac{C(u,u)}u\quad\text{ and }\quad\lambda_U(C)=\underset{u\to 1^-}\lim \frac{1-2u+C(u,u)}{1-u}.$$
In Durante et al. (2010), the authors study the effect of a distortion $\psi$ (given by a strictly increasing bijection of $[0,1]$) on the tail dependence coefficients of the distorted copula $C_\psi(u,v)=\psi\left (C\left(\psi^{-1}(u),\psi^{-1}(v)\right)\right )$. In particular, they analyze the case in which the behavior of $\psi(t)$ in the right neighborhood of $0$ and in the left neighborhood of $1$ is of power type, 
obtaining the following results:
\begin{enumerate}
\item if there exist $b,\alpha >0$ such that $\psi(z)\underset {t\downarrow 0}\sim bz^\alpha$, then  $\lambda_L(C_\psi)=\left (\lambda_L(C)\right )^\alpha$. 
\item if there exist $b,\alpha >0$ such that $1-\psi(z)\underset {t\uparrow 1}\sim b(1-z)^\alpha$, then  $\lambda_U(C_\psi)=2-\left (2-\lambda_U(C)\right )^\alpha$.
\end{enumerate}

In case the distortion $\psi$ decays to zero at an exponential speed, that is that there exist $a,\beta >0$ such that $\psi(z)\underset{t\downarrow 0}\sim e^{-az^{-\beta}}$ (this is the case of the distorion $h$ that allows to construct the bivariate Gompertz distribution in Marshall and Olkin, 2015, see Example \ref{MO15}), it can be easily verified that, if $\lambda_L(C)<1$, then $\lambda_L(C_\psi)=0$. In fact, 
$$\begin{aligned}\lambda_L(C_\psi)&=\underset{u\downarrow 0}\lim \frac{\psi\left (C\left(\psi^{-1}(u),\psi^{-1}(u)\right)\right )}{u}
=\underset{w\downarrow 0}\lim \frac{\psi\left (C(w,w)\right )}{\psi(w)}=
\underset{w\downarrow 0}\lim \frac{\exp\left (-aC^{-\beta}(w,w)\right )}{\exp (-aw^{-\beta})}=\\
&=\underset{w\downarrow 0}\lim \exp\left (-aC^{-\beta}(w,w)\left [1-\frac{C^\beta(w,w)}{w^\beta}\right ]\right )=0.
\end{aligned}$$

Since the family of distortions $d_t$, with $t\in [0,+\infty)$, can be expressed in terms of the generator $h$ (see (\ref{d})), we study the impact of the choice of $h$ on the time dependent tail coefficients.\medskip

In the sequel, with $C=C_0$ we denote the copula associated to $\bar F_0=\bar F$, while $C_{\bar G}$ is defined in (\ref{copula_G}).
\vskip 0.3cm
\noindent We start analyzing the case of $\lambda_L$.
\begin{lemma}\label{Lemma6}
\begin{enumerate}
\item If there exist $a,\beta>0$ such that $h(x)\underset{x\downarrow 0}\sim ax^\beta$,
then $d_t(z)\underset {z\downarrow 0}\sim bz$, for some $b>0$, implying that 
$\lambda_L(C_t)=\lambda_L(C)=\left (\lambda_L(C_{\bar G})\right )^\beta$.
\item If there exist $a,d,\beta>0$ such that
$h(x)\underset{x\downarrow 0}\sim de^{-ax^{-\beta}}$, then
$d_t(z)\underset {z\downarrow 0}\sim bz^{e^{\beta t}}$ for some $b>0$, implying that 
$\lambda_L(C_t)=\left (\lambda_L(C)\right )^{e^{\beta t}}$. If $\lambda_L(C_{\bar G})<1$, then  $\lambda_L(C_t)=\lambda_L(C)=0$.
\end{enumerate}\end{lemma}
\begin{proof}
In case {\it 1.}
$$\begin{aligned}
\underset{z\downarrow 0}\lim \frac{d_t(z)}{z}=
\underset{x\downarrow 0}\lim \frac 1{h(e^{-t})}\frac{h(e^{-t}x)}{h(x)}=\frac 1{h\left(e^{-t}\right)}\underset{x\downarrow 0}\lim \frac{ae^{-\beta t} x^\beta}{a x^{\beta}}=\frac {e^{-\beta t}}{h(e^{-t})}.
\end{aligned}$$

\noindent In case {\it 2.}

$$\begin{aligned}
\underset{z\downarrow 0}\lim \frac{d_t(z)}{z^{e^{\beta t}}}=
\underset{x\downarrow 0}\lim \frac 1{h\left(e^{-t}\right)}\frac{h\left (e^{-t}x\right )}{h^{e^{\beta t}}(x)}=\frac 1{h\left (e^{-t}\right)}\underset{x\downarrow 0}\lim \frac{e^{-ae^{\beta t}x^{-\beta}}}{d^{e^{\beta t} -1}e^{-e^{\beta t} ax^{-\beta}}}=\frac 1{d^{e^{\beta t} -1}h\left (e^{-t}\right )}.
\end{aligned}$$

\end{proof}

\begin{example}

Let us consider the survival distribution (\ref{MU_LMP}). The associated survival copula function is
\begin{equation}
\begin{split}
& C_{\bar G_{\alpha,\gamma,\alpha_1,\alpha_2}}(u,v)= \\
& = \left (\alpha_2\left (\frac{u^{-\alpha}-\alpha_1 }{1-\alpha_1}\right )+\alpha_1\left (\frac{v^{-\alpha}-\alpha_2 }{1-\alpha_2}\right )+(1-\alpha_1-\alpha_2)\max\left (\frac{u^{-\alpha}-\alpha_1 }{1-\alpha_1},\frac{v^{-\alpha}-\alpha_2 }{1-\alpha_2}\right)\right)^{-\frac 1\alpha}.
\end{split}
\label{C_alpha}
\end{equation}
If we assume $\alpha_1\geq \alpha_2$ (the opposite case is completely analogous), we get
$$\lambda_L(C_{\bar G_{\alpha,\gamma,\alpha_1,\alpha_2}})=\underset{u\to 0^+}\lim\frac{\left (u^{-\alpha}\left (1+\frac {\alpha_1}{1-\alpha_2}\right )-\frac{\alpha_1}{1-\alpha_2}\right )^{-\frac 1\alpha}}{u}=
\left(\frac{1-\alpha_2}{1+\alpha_1-\alpha_2}\right)^\frac{1}{\alpha}\in (0,1).$$
\begin{enumerate}
    \item In case 3. of Example \ref{ex_mix}, we have $h(x) = 1-\left(1-x^{\frac \gamma\alpha}\right)^a$, with $ a > 0$. Since $h(x)\underset {x\downarrow 0}\sim a x^{\frac \gamma\alpha}$, by Lemma \ref{Lemma6} and (\ref{mix_laplace}) we get  $\lambda_L(C_t) = \lambda_L(C) =\left ( \lambda_L(C_{\bar G_{\gamma,\gamma,\alpha_1,\alpha_2}})\right )^{\frac \gamma\alpha} =\left(\frac{1-\alpha_2}{1+\alpha_1-\alpha_2}\right)^\frac{\gamma}{\alpha ^2}$ for all $t\in [0,+\infty)$.
    \item In case 4. of Example \ref{ex_mix}, we have $h(x) = \frac{\log\left(1+\theta x^{\frac\gamma\alpha}\right)}{\log(1+\theta)}$, with $\theta \in (-1,0)$. Since
$h(x)\underset {x\downarrow 0}\sim \frac\theta{log(1+\theta)} x^{\frac \gamma\alpha}$, by Lemma \ref{Lemma6} and (\ref{mix_laplace}) we get 
 $\lambda_L(C_t) = \lambda_L(C) =\left ( \lambda_L(C_{\bar G_{\gamma,\gamma,\alpha_1,\alpha_2}})\right )^{\frac \gamma\alpha}  =\left(\frac{1-\alpha_2}{1+\alpha_1-\alpha_2}\right)^\frac{\gamma}{\alpha ^2}$ for all $t\in [0,+\infty)$.
\end{enumerate}
\end{example}

\vskip 0.3cm
Let us now analyze the upper tail dependence coefficient.
\begin{lemma}\label{Lemma 6.3} Let $h$ be differentiable in $(0,1)$ with $h^\prime$ continuous and $h^\prime(x)\in (0,+\infty)$.
If there exist $a,\beta>0$ such that $1-h(x)\underset{x\uparrow 1}\sim a(1-x)^\beta$, 
then $1-d_t(z)\underset {t\uparrow 1}\sim b(1-z)^{\frac 1\beta}$ for some $b>0$, implying that 
$\lambda_U(C_t)=2-\left(2- \lambda_U(C)\right )^{\frac 1\beta}=\lambda_U(C_{\bar G})$, for $t>0$.
\end{lemma}
\begin{proof}
The result immediately follows from
$$\begin{aligned}\underset{z\to 1^-}\lim \frac{1-d_t(z)}{(1-z)^{\frac 1\beta}}&=
\underset{x\to 1^-}\lim \frac{1-h_t(x)}{(1-h(x))^{\frac 1\beta}}=\frac 1{a^{\frac 1\beta}}\underset{x\to 1^-}\lim \frac{1-h_t(x)}{1-x}=\\
&=\frac{e^{-t}}{a^{\frac 1\beta} h\left (e^{-t}\right )}
\underset{x\to 1^-}\lim h^\prime\left(e^{-t}x\right )=\frac{e^{-t}h^\prime\left (e^{-t}\right )}{a^{\frac 1\beta} h\left(e^{-t}\right)}.\end{aligned}$$
\vskip 0.2cm
\end{proof}

%\begin{remark} In case {\it 1.} we have that $\lambda_U(\Bar{C}_{h_c}) = \lambda_U(\Bar{C^G})$.\end{remark}

\begin{example}

Let us consider the survival distribution (\ref{MU_LMP}). Using (\ref{C_alpha}), if we assume $\alpha_1 \geq \alpha_2$ then, 
\begin{equation*}
\begin{aligned}
\lambda_U(C_{\bar G_{\alpha,\gamma,\alpha_1,\alpha_2}})&= \underset{u \to 1} \lim \frac{1-2u+\bar{C}_{\bar G_\alpha}(u,u)}{1-u} = \underset{u \to 1} \lim \frac{1-2u+\left(u^{-\alpha}\left(1+\frac{\alpha_1}{1-\alpha_2}\right)-\frac{\alpha_1}{1-\alpha_2}\right)^{-\frac{1}{\alpha}}}{1-u} =\\
&= \frac{1-(\alpha_1+\alpha_2)}{1-\alpha_2}
\end{aligned}
\end{equation*}
\begin{enumerate}
    \item In case 3. of Example \ref{ex_mix}, we have $1-h(x) = \left(1-x^{\frac \gamma\alpha}\right)^a$, with $ a > 0$. Since $1-h(x)\underset {x\uparrow 1}\sim \left (\frac \gamma\alpha\right )^a(1-x)^a$, by Lemma \ref{Lemma 6.3} and (\ref{mix_laplace}) we get
$\lambda_U(C) = 2-\left (\frac{1+\alpha_1-\alpha_2}{1-\alpha_2}\right )^a$ and $\lambda_U(C_t) = \frac{1-(\alpha_1+\alpha_2)}{1-\alpha_2}$ for all $t\in (0,+\infty)$.
\item In case 4. of Example \ref{ex_mix}, we have $h(x) = \frac{\log(1+\theta x^{\frac\gamma\alpha})}{\log(1+\theta)}$, with $\theta \in (-1,0)$.
Since $1-h(x)\underset {x\uparrow 1}\sim \frac{\theta\gamma}{\alpha(1+\theta)\log(1+\theta)}(1-x)$, by Lemma \ref{Lemma 6.3} and (\ref{mix_laplace}) we get $\lambda_U(C_t) = \frac{1-(\alpha_1+\alpha_2)}{1-\alpha_2}$ for all $t\geq 0$.
\end{enumerate}
\end{example}
\section{Conclusions}\label{conclusion}
In this note, we have generalized the functional equations that characterize the lack-of-memory properties in survival analysis (at the univariate as well as at the bivariate levels) by assuming that the survival distribution of the residual lifetimes is given by a time dependent distortion of that of the original lifetimes: these equations represent a generalization of the univariate functional equation introduced by Kaminsky (1983) and of the bivariate strong and weak versions studied by Marshall and Olkin (2015).  After determining the conditions under which they have solutions, we show that they are equivalent to those studied in Ricci (2024). \\
Since the univariate case turns out to be trivial and the distributions that satisfy the generalized strong bivariate equation have already been studied in the literature, we have focused our analysis on the generalized weak bivariate case, where residual lifetimes are conditioned on survival beyond a common threshold. \\
Through a mixing approach, we have generated new classes of bivariate survival distributions starting from a family of distributions that have been studied in Mulinacci (2018) and assuming a positive mixing variable whose moment generating function is known in closed form. \\
Since the functional equation is characterized by the given time dependent distortion that links the survival distribution of the residual lifetimes to the original one, we have analyzed its impact on the bivariate aging properties and on the dependence structure of the residual lifetimes. In particular, we have shown, through many examples, that the choice of the distortion (that turns out to depend on the distribution of the minimum of the two involved lifetimes) allows to build bivariate distributions that may exhibit bivariate increasing (decreasing) failure rates or New Better (Worse) than Used properties. By analyzing the time-dependent Kendall’s function and the tail dependence coefficients, we have shown how the strength and nature of dependence can vary—either intensifying or weakening—as the conditioning time increases: in particular, we have proved that upper tail dependence coefficient may have a discontinuity at time $0$ for a specific choice of the generator.\\
Finally, we have discussed possible applications to joint tail risk management and insurance pricing. As for the latter, our simulations on the pricing of joint survivor annuities demonstrate that, accounting for dynamic and possibly asymmetric dependence, yields substantial differences in annuity values compared to the standard assumption of independence: depending on whether dependence is positive or negative, and on whether it increases or decreases over time,  the premium can be significantly higher or lower. These findings reinforce the practical relevance of incorporating more realistic dependence structures in actuarial models. \\
Future work may extend this framework to higher dimensions or explore estimation procedures based on real-world data.


\begin{thebibliography}{99}
%\bibitem{BB} H.W. Block, A.P. Basu. "A continuous bivariate exponential extension". In: Journal of the American Statistical Association (1974), Volume 69, pp. 1031-1037.
\bibitem{BFL} M. Denuit, E. Frostig and B. Levikson. "Shift in interest rate and common shock model
for coupled lives". In: Belgian Actuarial Bulletin (2006) , Volume 6, pp. 1–4.
\bibitem{D} F. Durante, R. Foschi, P. Sarkoci. "Distorted Copulas: Constructions and Tail Dependence". In : Communications in Statistics
- Theory and Methods (2010), Volume 39 (12), pp. 2288-2301.
\bibitem{GK} C. Genest, N. Kolev. " A law of uniform seniority for dependent lives". In: 
  Scandinavian Actuarial Journal (2021), Volume 2021, Issue 8,  pp. 726-743.
\bibitem{GL} C. Gourieroux, Y. Lu. " Love and death: A Freund model with frailty". In: 
Insurance: Mathematics and Economics (2015), Volume 63, pp. 191-203.
\bibitem{J} H. Joe. "Dependence modeling with copulas". CRC press. (2014).
\bibitem{K} K.S. Kaminsky. "An aging property of the Gompertz survival function and a discrete analog". Tech. Rep., Department of Mathematical Statistics,
University of Umeå, Sweden, 1983.
\bibitem{Kl} E.P. Klement, R. Mesiar, E. Pap. "Triangular norms. Position paper III: continuous t-norms"
In: Fuzzy Sets and Systems (2004), Volume 145(3), pp. 439-454.
\bibitem{Kol} N. Kolev. "Characterizations of the class of bivariate Gompertz distributions"
In: Journal of Multivariate Analysis (2016), Volume 148, pp. 173-179.
%\bibitem {KU} H.V.Kulkarni. "Characterizations and modelling of multivariate lack-of-memory property". In: Metrika (2006), Volume 64, pp.167-180.
%\bibitem{LP} X.Li, F.Pellerey. "Generalized Marshall-Olkin distributions and related bivariate aging properties". In: Journal of Multivariate Analysis (2011), Volume 102, pp. 1399-1409.
\bibitem{MO} A.Marshall I.Olkin. "A multivariate exponential distribution". In: Journal of the American
Statistical Association (1967), Volume 62, pp. 30–44.
\bibitem{MO2} A.W.Marshall, I.Olkin. "A Bivariate Gompertz-Makeham Life Distribution". In: Journal of Multivariate Analysis (2015), Volume 139, pp. 219-226.
\bibitem{MO3} A.W.Marshall, I.Olkin. "Life Distributions". Springer Series in Statistics (2007).
\bibitem{MS} P.Muliere, M.Scarsini. "Characterization of a Marshall-Olkin Type Class of Distributions". In: Annals of the Institute of Statistical Mathematics (1987), Volume 39, pp. 429-441.
\bibitem{SM} S.Mulinacci. "Archimedean-based Marshall-Olkin Distributions and Related Dependence Structures". In: Methodology and Computing in Applied Probability (2018), Volume 20, pp. 205-236. 
\bibitem{NQRU} R.B. Nelsen, J.J. Quesada-Molina, J.A. Rodríguez-Lallena, M. Úbeda-Flores: "Kendall distribution functions". In:
Stat Probab Lett (2003), Volume 65(3), pp. 263–268.
\bibitem{N} R. B. Nelsen. "An Introduction to Copulas". Springer, 2nd Edn (2006). 
%\bibitem{PK} J.Pinto, N.Kolev. "Sibuya-type bivariate lack-of-memory property". In: Journal of Multivariate Analysis, (2015), Volume 134, pp. 119-128.
\bibitem{R} B. L. S. P. Rao. "On a Bivariate lack-of-memory Property Under Binary Associative Operation". In: COMMUNICATIONS IN STATISTICS, Theory and Methods (2004), Volume 33, Number 12, pp. 3103–3114.
\bibitem{RM} M. Ricci. "A Generalization of Bivariate Lack-of-Memory Properties". In: https://arxiv.org/abs/2401.11457 (2024). 
\end{thebibliography}
\end{document}